\documentclass{amsart}
\usepackage{amsfonts}

\newtheorem{theorem}{Theorem}
\theoremstyle{plain}

\newtheorem{definition}{Definition}

\newtheorem{proposition}{Proposition}
\newtheorem{remark}{Remark}

\numberwithin{equation}{section}
\input{tcilatex}

\begin{document}
\title[Hurwitz function]{On connection between values of Riemann zeta
function at integers and generalized harmonic numbers }
\author{Pawe\l\ J. Szab\l owski}
\address{Department of Mathematics and Information Sciences,\\
Warsaw University of Technology\\
ul. Koszykowa 75, 00-662 Warsaw, Poland}
\email{pawel.szablowski@gmail.com}
\date{December, 2014}
\subjclass[2010]{Primary 11M35, 40G05; Secondary 05A15, 05E05}
\keywords{Riemann zeta function Hurwitz zeta function, Euler summation,
Harmonic numbers, generalized Harmonic numbers, Catalan constant. }

\begin{abstract}
Using Euler transformation of series we relate values of Hurwitz zeta
function at integer and rational values of arguments to certain rapidly
converging series where some generalized harmonic numbers appear. The form
of these generalized harmonic numbers carries information about the values
of the arguments of Hurwitz function. In particular we prove: $\forall k\in 
\mathbb{N}:$ $\zeta (k,1)\allowbreak =\allowbreak \frac{2^{k-1}}{2^{k-1}-1}%
\sum_{n=1}^{\infty }\frac{H_{n}^{(k-1)}}{n2^{n}},$ where $H_{n}^{(k)}$ are
defined below generalized harmonic numbers. Further we find generating
function of the numbers $\hat{\zeta}(k)=\sum_{j=1}^{\infty }(-1)^{j-1}/j^{k}.
$
\end{abstract}

\maketitle

\section{Introduction}

First let us recall basic notions and definitions that we will work with. By
the Hurwitz function $\zeta (s,\alpha )$ we will mean:%
\begin{equation*}
\zeta (s,\alpha )\allowbreak =\allowbreak \sum_{j=0}^{\infty }\frac{1}{%
(j+\alpha )^{s}},
\end{equation*}%
considered for $\func{Re}s>1,$ $\func{Re}\alpha \in (0,1].$ Function $\zeta
(s,1)$ is called Riemann zeta function. We will denote it also by $\zeta (s)$
if it will not cause misunderstanding. It turns out that both these
functions can be extended to holomorphic functions of $s$ on the whole
complex plane except $s=1$ where a single pole exists. Of great help in
doing so is the formula 
\begin{equation}
\zeta (s)\allowbreak =\allowbreak \frac{2^{s-1}}{2^{s-1}-1}%
\sum_{j=1}^{\infty }\frac{(-1)^{j-1}}{j^{s}},  \label{rz}
\end{equation}%
that enables to extend Riemann zeta function to the whole half plane $\func{%
Re}s>0$.

We will consider numbers:%
\begin{equation*}
\mathbb{M}_{k}^{(m,i)}\allowbreak =\allowbreak \sum_{j=0}^{\infty }\frac{%
(-1)^{j}}{(mj+i)^{k}},
\end{equation*}%
for $m\in \mathbb{N}$ and $i\in \left\{ 1,\ldots ,m-1\right\} .$ Notice that 
$\mathbb{M}_{k}^{(1,1)}\allowbreak =\allowbreak \sum_{j=1}^{\infty
}(-1)^{j-1}/j^{k}$ and $\mathbb{M}_{1}^{(2,1)}\allowbreak =\allowbreak \pi
/4.$ The number $\mathbb{M}_{2}^{(2,1)}\allowbreak =\allowbreak
\sum_{j=0}^{\infty }\frac{(-1)^{j}}{(2j+1)^{2}}$ is called Catalan constant $%
\mathbb{K}.$

It is elementary to notice that 
\begin{equation*}
\mathbb{M}_{k}^{(m,i)}\allowbreak =\allowbreak \frac{1}{(2m)^{k}}(\zeta
(k,i/(2m))-\zeta (k,1/2+i/(2m)).
\end{equation*}

The main idea of this paper is to apply the so called Euler transformation
that was nicely recalled in \cite{snow94}. As pointed out there we have:%
\begin{equation*}
\sum_{k=1}^{\infty }(-1)^{k-1}a_{k}\allowbreak =\allowbreak
\sum_{n=0}^{\infty }\Delta ^{n}a_{1}/2^{n+1},
\end{equation*}%
where $\left\{ a_{k}\right\} _{k\geq 1}$ is a sequence of complex numbers
and the sequence $\Delta ^{n}a_{k}$ is defined recursively: $\Delta
^{0}a_{k}\allowbreak =\allowbreak a_{k,}$ $\Delta ^{n}a_{k}\allowbreak
=\allowbreak \Delta ^{n-1}a_{k}\allowbreak -\allowbreak \Delta
^{n-1}a_{k+1}\allowbreak \allowbreak =\allowbreak \sum_{m=0}^{n}(-1)^{m}%
\binom{n}{m}a_{m+k}.$ Sondow in \cite{snow94} presented general idea of
applying Euler transformation to Riemann function. He however stopped half
way in the sense that he calculated finite differences $\Delta ^{n}$ applied
to $(j+1)^{-s}$ only for $s$ being negative integers. We are going to make a
few steps further and calculate these differences pointing out the r\^{o}le
of the generalized harmonic numbers in those calculations.

The paper is organized as follows. In the next section \ref{s_eul} we
present an auxiliary result that enables application of Euler transformation
to the analyzed series. Further we present transformed series approximating
numbers $\mathbb{M}_{k}^{(m,i)}$.  In Section \ref{s_gen} we calculate
generating functions of certain series of numbers and functions. More
precisely we calculate generating functions of the generalized harmonic
numbers that we have defined in the previous section. We also calculate
generating function of the series of the generating functions that were
defined previously. It turns out that this calculation enables to obtain the
generating function of the series sums that appear on the right hand side of
(\ref{rz}). Finally in the last Section \ref{remark} are collected cases
when exact values of numbers $\mathbb{M}$ are known.

\section{Euler transformation\label{s_eul}}

To proceed further we need the following result.

\begin{proposition}
\label{euler}Let us denote $A_{n,k}^{(m,i)}\allowbreak =\allowbreak
\sum_{j=0}^{n}(-1)^{j}\binom{n}{j}/(mj+i)^{k}$, $n\allowbreak =\allowbreak
0,1,\ldots ,$ and the family of sequences defined recursively: $%
B_{n,0}^{(m,i)}\allowbreak =\allowbreak 1,$ $B_{0,k}\allowbreak =\allowbreak 
\frac{1}{i^{k-1}},$ $k\geq 1,$ $\forall n,k\geq 0:B_{n,k}^{(m,i)}\allowbreak
=\allowbreak \allowbreak \sum_{j=0}^{n}\frac{1}{(mj+i)}B_{j,k-1}^{(m,i)}$.
We have then:\newline
$\forall $ $m\in \mathbb{N}$ : $A_{0,0}^{(m,i)}\allowbreak =1$, $\allowbreak
A_{n,0}^{(m,i)}\allowbreak =0,$ $\allowbreak A_{n,1}^{(m,i)}\allowbreak =%
\frac{n!}{m(i/m)_{n+1}},$ where \newline
$(a)_{n}\allowbreak =\allowbreak a(a+1)\ldots (a+n-1)$ is the so called
'rising factorial'. \newline
$\forall $ $n\geq 0,k\geq 1$ we get: 
\begin{equation*}
A_{n,k}^{(m,i)}\allowbreak =\allowbreak \frac{n!}{m(i/m)_{n+1}}%
B_{n,k-1}^{(m,i)}.
\end{equation*}
\end{proposition}

\begin{proof}
i) The fact that $A_{n,0}\allowbreak =\allowbreak 0$ follows immediately
properties of binomial coefficients. Notice that we have 
\begin{gather*}
A_{n+1,k}^{(m,i)}-\frac{m(n+1)}{m(n+1)+i}A_{n,k}^{(m,i)}=%
\sum_{j=0}^{n}(-1)^{j}\binom{n+1}{j}/(mj+i)^{k}+\frac{(-1)^{n+1}}{\left(
m(n+1)+i\right) ^{k}} \\
-\frac{m(n+1)}{m(n++1)+i}\sum_{j=0}^{n}(-1)^{j}\binom{n}{j}/(mj+i)^{k}=\frac{%
(-1)^{n+1}}{\left( m(n+1)+i\right) ^{k}} \\
+\sum_{j=0}^{n}(-1)^{j}(\binom{n+1}{j}-\frac{m(n+1)!}{j!(n-j)!(m(n+1)+i)}%
)/(mj+i)^{k} \\
=\frac{(-1)^{n+1}}{\left( m(n+1)+i\right) ^{k}}+\frac{1}{m(n+1)+j}%
\sum_{j=0}^{n}(-1)^{j}\binom{n+1}{j}/(mj+i)^{k-1} \\
=\frac{1}{m(n+1)+i}\sum_{j=0}^{n+1}(-1)^{j}\binom{n+1}{j}/(mj+i)^{k-1}%
\allowbreak =\allowbreak \frac{1}{m(n+1)+i}A_{n+1,k-1}^{(m,i)}
\end{gather*}%
since $(1-\frac{m(n+1-j)}{(mn+m+i)})\allowbreak =\frac{jm+i}{m+mn+i}$. Now
notice that we have $A_{n+1,1}^{(m.i)}-\frac{m(n+1)}{m(n+1)+i}%
A_{n,1}^{(m,i)}\allowbreak =\allowbreak 0$ from which immediately follows
that $A_{n,1}^{(m,i)}\allowbreak =\allowbreak \frac{n!}{m(i/m)_{n+1}}$ since 
$A_{0,1}^{(m,i)}\allowbreak =\allowbreak \frac{1}{i}$. Now divide both sides
of the identity $A_{n+1,k}^{(m,i)}-\frac{m(n+1)}{mn+m+i}A_{n,k}^{(m,i)}%
\allowbreak =\allowbreak \frac{1}{mn+m+i}A_{n+1,k-1}^{(m,i)}$ by $%
A_{n+1,1}^{(m,i)}$ and denote $B_{n,k}^{(m,i)}\allowbreak =\allowbreak
A_{n,k}^{(m,i)}/A_{n,1}^{(m,i)}.$ We get $B_{n+1,k}^{(m,i)}-B_{n,k}^{(m,i)}%
\allowbreak =\allowbreak \frac{1}{m(n+1)+i}B_{n+1,k-1}^{(m,i)}$ since $%
A_{n+1,1}^{(m.i)}\allowbreak =\allowbreak \frac{m(n+1)}{mn+m+i}%
A_{n,1}^{(m,i)}.$ Hence $B_{n,k}^{(m,i)}\allowbreak =\allowbreak
\sum_{j=0}^{n}\frac{1}{mj+i}B_{j,k-1}^{(m,i)}$ since $\forall k\geq
1:B_{0,k}^{(n,i)}\allowbreak =\allowbreak 1/i^{k-1}.$
\end{proof}

\begin{remark}
In the literature (compare e.g. \cite{choi2011}, \cite{choi2014}, \cite%
{Kron12}) there function notions of harmonic and generalized harmonic
numbers defined by $h_{n}^{(k)}\allowbreak =\allowbreak
\sum_{j=1}^{n}1/j^{k} $, $n\geq 1.$ Numbers $h_{n}^{(1)}$ are called simply
(ordinary) harmonic numbers.
\end{remark}

We are going to define differently generalized harmonic numbers.

\begin{definition}
\label{harm}For every $k\in \mathbb{N}$ numbers $\left\{ H_{n}^{(k)}\right\}
_{n\geq 1,k\geq 0}$ defined recursively by $H_{n}^{(0)}\allowbreak
=\allowbreak 1,$ $H_{n}^{(k)}\allowbreak =\allowbreak
\sum_{j=1}^{n}H_{j}^{(k-1)}/j,$ $n\geq 1$ will be called generalized
harmonic numbers of order $k.$
\end{definition}

\begin{remark}
It is easy to see that $B_{n,k}^{(1,1)}\allowbreak =\allowbreak
H_{n+1}^{(k)} $ and that $H_{n}^{(1)}$ is an ordinary $n-$th harmonic number.
\end{remark}

\begin{remark}
Notice that $H_{n}^{(k)}$ is a symmetric function of order $k$ of the
numbers $\{1,1/2,\ldots ,1/n\}$ hence it can be expressed as a linear
combination of some other symmetric functions of order less or equal $k.$
For example we have: $H_{n}^{(1)}\allowbreak =\allowbreak
h_{n}^{(1)}\allowbreak =\allowbreak H_{n}$ (the ordinary harmonic number), $%
H_{n}^{(2)}\allowbreak =\allowbreak H_{n}^{2}/2\allowbreak +\allowbreak
h_{n}^{(2)}/2,$ $H_{n}^{(3)}\allowbreak =\allowbreak H_{n}^{3}/6\allowbreak
+\allowbreak H_{n}h_{n}^{(2)}/2\allowbreak +\allowbreak h_{n}^{(3)}/3$ and
so on.
\end{remark}

\begin{remark}
Notice also that recursive equation that was obtained in the proof of
Proposition \ref{euler} i.e.%
\begin{equation*}
A_{n+1,k}^{(m,i)}-\frac{m(n+1)}{m(n+1)+i}A_{n,k}^{(m,i)}\allowbreak
=\allowbreak \frac{1}{m(n+1)+i}A_{n+1,k-1},
\end{equation*}%
is valid also for $k\allowbreak =\allowbreak 0,-1,-2,\ldots .$ . Of course
then we apply it in the following form:%
\begin{equation*}
A_{n+1,k-1}=(m(n+1)+i)A_{n+1,k}\allowbreak -\allowbreak
m(n+1)A_{n,k}^{(m,i)}\allowbreak ,
\end{equation*}%
getting for example : $A_{0,-1}^{(m,i)}\allowbreak =\allowbreak 1,$ $%
A_{1,-1}^{(m,i)}\allowbreak =\allowbreak -m,$ $A_{n,-1}^{(m,i)}\allowbreak
=\allowbreak 0,$ $A_{0,-2}^{(m,i)}\allowbreak =\allowbreak 1,$ $%
A_{1,-2}^{(m,i)}\allowbreak =\allowbreak -m(m+i+1),$ $A_{2,-2}^{(m,i)}%
\allowbreak =\allowbreak 2m^{2}$ , $A_{n,-2}^{(m,i)}\allowbreak =\allowbreak
0$ for $n=3,4,\ldots $ . The fact that $A_{n,-k}^{(m,i)}\allowbreak
=\allowbreak 0$ for $n\geq k+1$ was already noticed, justified and applied
by Sondow in \cite{snow94}.
\end{remark}

As a corollary we have the following result:

\begin{theorem}
\begin{equation}
\mathbb{M}_{k}^{(m,i)}\allowbreak =\allowbreak \sum_{n=0}^{\infty }\frac{n!}{%
2^{n+1}m(i/m)_{n+1}}B_{n,k-1}^{(m,i)},  \label{Mk}
\end{equation}%
where numbers $B_{n,k}^{(m,i)}$ are defined above.

i) In particular :%
\begin{equation}
\mathbb{M}_{2k+1}^{(2m,m)}=\frac{1}{m^{2k+1}}M_{2k+1}^{(2,1)}=\pi ^{2k+1}%
\frac{(-1)^{k}E_{2k}}{2(2m)^{2k+1}(2k)!},  \label{eul}
\end{equation}
\begin{equation}
\mathbb{M}_{2}^{(2,1)}\allowbreak =\allowbreak \mathbb{K\allowbreak
=\allowbreak }\sum_{n=0}^{\infty }\frac{n!(H_{2n+1}-H_{n}/2)}{2(2n+1)!!},
\label{cat}
\end{equation}%
where $H_{n}$ denotes $n-th$ (ordinary) harmonic number.

ii) for $m\allowbreak =\allowbreak i\allowbreak =\allowbreak 1,$ $k\in 
\mathbb{N}\mathbf{:}$%
\begin{equation}
\sum_{j=1}^{\infty }\frac{(-1)^{j-1}}{j^{k}}\allowbreak =\allowbreak
\sum_{n=1}^{\infty }\frac{H_{n}^{(k-1)}}{n2^{n}},  \label{riem1}
\end{equation}%
and consequently for $k\allowbreak =\allowbreak 2,3,\ldots $%
\begin{equation}
\zeta (k)\allowbreak =\allowbreak \frac{2^{k-1}}{2^{k-1}-1}%
\sum_{n=1}^{\infty }\frac{H_{n}^{(k-1)}}{n2^{n}}.  \label{riem}
\end{equation}
\end{theorem}

\begin{proof}
Applying Euler transformation to the series $\mathbb{M}_{n,k}^{(m,i)}$ we
have 
\begin{equation*}
\mathbb{M}_{n,k}^{(m,i)}\allowbreak =\allowbreak \sum_{n=0}^{\infty
}A_{n,k}^{(m,i)}/2^{n+1}.
\end{equation*}%
Now it remains to apply Proposition \ref{euler}. i) To see that (\ref{Mk})
reduces to (\ref{cat}) when $k\allowbreak =\allowbreak 2,$ $m\allowbreak
=\allowbreak 2$ and $i\allowbreak =\allowbreak 1$ notice that $%
B_{n,0}^{(2,1)}\allowbreak =\allowbreak 1$ and consequently $%
B_{n,1}^{(2,1)}\allowbreak =\allowbreak \sum_{j=0}^{n}1/(2j+1)\allowbreak
=\allowbreak H_{2n+1}-2H_{n}.$ Further we have $(1/2)_{n+1}\allowbreak
=\allowbreak \prod_{j=0}^{n}(j+1/2)\allowbreak =\allowbreak (2n+1)!!/2^{n+1}.
$ To justify (\ref{eul}) we have to observe that $2\mathbb{M}%
_{2k+1}^{(2m,m)}\allowbreak =\allowbreak \hat{S}(2k+1,2m,m)\allowbreak
=\allowbreak \sum_{j=-\infty }^{\infty }\frac{(-1)^{j}}{(j2m+m)^{2k+1}}%
\allowbreak =\allowbreak \frac{1}{m^{2k+1}}\sum_{j=-\infty }^{\infty }\frac{%
(-1)^{j}}{(j2+1)^{2k+1}}.$ The fact that $\sum_{j=-\infty }^{\infty }\frac{%
(-1)^{j}}{(j2+1)^{2k+1}}\allowbreak =\allowbreak \pi ^{2k+1}(-1)^{k}\frac{%
E_{2k}}{2^{2k+1}(2k)!}$ dates back to Euler and was recalled in \cite{Szab14}%
.

ii) If $m\allowbreak =\allowbreak i\allowbreak =\allowbreak 1$ we have $%
(1)_{n+1}\allowbreak =\allowbreak (n+1)!.$ Recall also that then $%
B_{n,k}^{(m,i)}\allowbreak =\allowbreak H_{n+1}^{(k)}.$  (\ref{riem})
follows additionally (\ref{rz}).
\end{proof}

\begin{remark}
Notice that when $i\allowbreak =\allowbreak 1$ then the sequence $\left\{
B_{n,k}^{(m,1)}\right\} $ is generated by the recursion: $%
B_{n,0}^{(m,1)}\allowbreak =\allowbreak 1,$ $B_{n,k}^{(m,1)}\allowbreak
=\allowbreak \sum_{j=0}^{n}B_{n.k-1}^{(m,1)}/(mj+1).$ Now arguing by
induction we see that $\forall n\geq 0$: $B_{n,k}^{(m,1)}\geq
B_{n,k-1}^{(m,1)}.$ Consequently we deduce that the sequence $\left\{ 
\mathbb{M}_{k}^{(m,1)}\right\} _{k\geq 1}$ is increasing which is not so
obvious when considering only definition of these numbers. It is also
elementary to notice that 
\begin{equation*}
\lim_{k\longrightarrow \infty }\mathbb{M}_{k}^{(m,1)}\allowbreak
=\allowbreak 1.
\end{equation*}%
In particular we deduce that the sequence $\{\zeta (k)(1-1/2^{k-1})\}_{k>1}$
is increasing.
\end{remark}

\begin{remark}
Notice that one can easily prove (by induction) that $\forall n,k\in \mathbb{%
N}:1\leq H_{n}^{(k)}\leq n.$ Hence, utilizing (\ref{riem}) we have:%
\begin{equation*}
\frac{\ln 2-1/2}{2^{m+1}(m+1)}\leq \left\vert \zeta (k)-\frac{2^{k-1}}{%
2^{k-1}-1}\sum_{n=0}^{m}\frac{H_{n+1}^{(k-1)}}{2^{n+1}(n+1)}\right\vert \leq 
\frac{1}{2^{m+1}},
\end{equation*}%
since $\frac{2^{k-1}}{2^{k-1}-1}\leq 2$ for $k\leq 2$ and further $%
\left\vert \zeta (k)-\frac{2^{k-1}}{2^{k-1}-1}\sum_{n=0}^{m}\frac{%
H_{n+1}^{(k-1)}}{2^{n+1}(n+1)}\right\vert \allowbreak \leq $\newline
$\allowbreak \frac{2^{k-1}}{2^{k-1}-1}\sum_{n=m+1}^{\infty
}1/2^{n+1}\allowbreak \leq \allowbreak \frac{2^{k-1}}{2^{k-1}-1}/2^{m+2}$
and $\frac{m+1}{n+1}\geq \frac{1}{n-m+1}$ and $\sum_{n=m+1}^{\infty }\frac{1%
}{2^{n+1}(n+1)}\allowbreak \geq \allowbreak \frac{1}{2^{m+1}(m+1)}%
\sum_{n=m+1}^{\infty }\frac{1}{2^{n-m+1}(n-m+1)}\allowbreak =\allowbreak 
\frac{\ln 2-1/2}{2^{m+1}(m+1)}$.
\end{remark}

\begin{remark}
Formulae (\ref{Mk}) and (\ref{riem}) can be considered as a series
transformation to speed up its convergence. Apery for $\zeta (3)$ in his
breakthrough paper and later Hessami Pilehrood et al. in \cite{hess11}
obtained series transformations to speedup series appearing in the
definitions of Riemann or Hurwitz zeta functions. As it is remarked in \cite%
{hess11} all these transformation give series more or less of the form $%
c_{n}/4^{n}$ where $c_{n}\allowbreak =\allowbreak O(1),$ but for different
arguments of $\zeta $ one gets very different series in a very different,
particular way. Apery's one is one of the simplest. Formulae (\ref{Mk}) and (%
\ref{riem}) offer unified form of the transformed series and speed of
convergence is only slightly worse. Namely of the form $c_{n}/2^{n}.$
\end{remark}

\begin{remark}
Notice also that analyzing the proof of Proposition \ref{euler} we can
formulate the following observation. Let us denote $A_{n,s}^{(m,l)}%
\allowbreak =\allowbreak \sum_{j=0}^{n}(-1)^{j}\binom{n}{j}/(mj+l)^{s}$ for $%
\func{Re}(s)>0.$ Then 
\begin{equation*}
A_{n+1,s}^{(m,l)}-\frac{m(n+1)}{m(n+1)+l}A_{n,s}^{(m,l)}\allowbreak
=\allowbreak \frac{1}{m(n+1)+l}A_{n+1,s-1}^{(m,l)}.
\end{equation*}%
Hence keeping in mind that $A_{0,s}^{(m,l)}\allowbreak =\allowbreak 1/l^{s}$
and assuming that we know numbers $\left\{ A_{n,s-1}^{(m,l)}\right\} _{n\geq
0}$ we are able to get numbers $\left\{ A_{n,s}^{(m,l)}\right\} _{n\geq 0}$
and consequently find $\zeta (s,l/m).$

In particular if $m\allowbreak =\allowbreak l\allowbreak =\allowbreak 1$ we
get $A_{n+1,s}\allowbreak -\allowbreak \frac{n+1}{n+2}\allowbreak
A_{n,s}=\allowbreak \frac{1}{n+2}A_{n+1,s-1}$ where we denoted $%
A_{n,s}\allowbreak =\allowbreak A_{n,s}^{(1,1)}$ to simplify notation.
Consequently we deduce that $A_{n,s}\allowbreak =\allowbreak \frac{1}{n+1}%
\sum_{j=1}^{n}A_{j,s-1}.$ Since we can iterate this relationship we see that
the knowledge of functions $A_{n,s}$ for $\func{Re}(s)\in (0,1]$ implies
knowledge of these functions for $s$ with $\func{Re}(s)>0.$
\end{remark}

\section{Generating functions and integral representation Riemann zeta
functions at integer values\label{s_gen}}

Let us denote by $f_{n}(x)$ the generating function of numbers $\left\{
H_{j}^{(n)}\right\} _{j=0}^{\infty }$ i.e. $f_{n}(x)\allowbreak =\allowbreak
\sum_{j=0}^{\infty }x^{j}H_{j+1}^{(n)}.$ We have the following simple
observation:

\begin{proposition}
i) $\forall x\in (-1,1):$ $f_{-1}(x)\allowbreak =\allowbreak 1,$ $%
f_{0}(x)\allowbreak =\allowbreak 1/(1-x):$%
\begin{equation}
f_{n}(x)\allowbreak =\allowbreak \frac{1}{x(1-x)}\int_{0}^{x}f_{n-1}(y)dy,
\label{gen}
\end{equation}%
$n\geq 1.$

ii) Let us denote $Q(x,y)$ the generating function of function series $%
\left\{ f_{n}\right\} _{n\geq 0}$ i.e. $Q(x,y)\allowbreak =\allowbreak
\sum_{j=0}^{\infty }y^{j}f_{j}(x),$ for $y\in (-1,1).$ We have 
\begin{equation}
Q(x,y)\allowbreak =\allowbreak \frac{B(x,1-y,1+y)}{x^{1-y}(1-x)^{1+y}},
\label{gen1}
\end{equation}%
where $B(x,a,b)$ denotes incomplete beta function.
\end{proposition}

\begin{proof}
i) We have $f_{n}(x)\allowbreak =\allowbreak \sum_{j=1}^{\infty
}x^{j-1}H_{j}^{(n)}\allowbreak =\allowbreak \sum_{j=1}^{\infty
}x^{j-1}\sum_{k=1}^{j}H_{k}^{(n-1)}/k\allowbreak =$\newline
$\allowbreak \sum_{k=1}^{\infty }H_{k}^{(n-1)}/k\sum_{j=k}^{\infty
}x^{j-1}\allowbreak =\allowbreak \frac{1}{1-x}%
\sum_{k=1}x^{k-1}H_{k}^{(n-1)}/k\allowbreak =\allowbreak $\newline
$\allowbreak \frac{1}{x(1-x)}\sum_{k=1}^{\infty
}H_{k}^{(n-1)}\int_{0}^{s}y^{k-1}dy\allowbreak =\allowbreak \allowbreak 
\frac{1}{x(1-x)}\int_{0}^{x}\sum_{k=1}^{\infty
}y^{k-1}H_{k}^{(n-1)}dy\allowbreak $\newline
$=\allowbreak \frac{1}{x(1-x)}\int_{0}^{x}f_{n-1}(y)dy.$

ii) We have: $(1-x)xQ(x,y)\allowbreak \allowbreak =\allowbreak
\sum_{j=0}^{\infty }y^{j}(1-x)xf_{j}(x)\allowbreak =\allowbreak
x+\sum_{j=1}^{\infty }y^{j}\int_{0}^{x}f_{j-1}(z)dz\allowbreak =$\newline
$\allowbreak x+\allowbreak \int_{0}^{x}\sum_{j=1}^{\infty
}y^{j}f_{j-1}(z)dz)\allowbreak =x+\allowbreak y\int_{0}^{x}Q(z,y))dz.$
Differentiating with respect to $x$ we get: $(1-2x)Q(x,y)\allowbreak
+\allowbreak x(1-x)Q^{\prime }(x,y)\allowbreak \allowbreak \allowbreak
=\allowbreak 1+yQ(x,y).$ Now solving this differential equation we get $%
Q(x,y)\allowbreak =\allowbreak \frac{Beta(x,1-y,1+y)-C(y)}{x^{1-y}(1-x)^{1+y}%
}\allowbreak .$ Recalling that $Q(0,y)\allowbreak =\allowbreak 1/(1-y)$ we
see $C(y)\allowbreak =\allowbreak 0$.
\end{proof}

Let us denote for simplicity $\hat{\zeta}(s)\allowbreak \overset{df}{=}%
\allowbreak \sum_{j=1}^{\infty }(-1)^{j}/j^{s}$ for $\func{Re}(s)>0.$ Notice
that following (\ref{riem1}) we have%
\begin{equation}
\hat{\zeta}(k)\allowbreak =\allowbreak \int_{0}^{1/2}f_{k-1}(x)dx\allowbreak
=\allowbreak \frac{1}{4}f_{k}(1/2),  \label{val}
\end{equation}%
for $k\allowbreak =\allowbreak 1,2,\ldots $ . 

We also have:%
\begin{equation*}
\sum_{j=0}^{\infty }y^{j}\hat{\zeta}(j)=B(1/2,1-y,1+y),
\end{equation*}%
for $y\in (-1,1)$ following (\ref{gen1}).

Recall that $\sum_{j=1}^{\infty }\zeta (2j)t^{2j}\allowbreak =\allowbreak
1-\pi t\cot (\pi t)$ hence $\sum_{j=1}^{\infty }\hat{\zeta}%
(2j)t^{2j}\allowbreak =\allowbreak \allowbreak \frac{\pi t}{\sin (\pi t)}%
\allowbreak -\allowbreak 1$ after some algebra.  Hence 
\begin{equation*}
\sum_{j=0}^{\infty }y^{2j+1}\hat{\zeta}(2j+1)\allowbreak =\allowbreak
B(1/2,1-y,1+y)+1-\frac{\pi y}{\sin (\pi y)},
\end{equation*}%
since $\hat{\zeta}(0)\allowbreak =\allowbreak 1/2.$ Let us remark that there
exist some expansions of incomplete beta function. Applying one of them one
can we have for example:%
\begin{equation*}
\sum_{j=0}^{\infty }y^{j}\hat{\zeta}(j)\allowbreak =\allowbreak
2^{y-1}\sum_{j=0}^{\infty }\frac{\left( -y\right) _{j}}{j!(j+1-y)2^{j}},
\end{equation*}%
for $y\in (0,1).$

\section{Remarks on particular values\label{remark}}

In \cite{Szab14} the sums of the form $S(n,k,l)\allowbreak =\allowbreak
\sum_{j=-\infty }^{\infty }\frac{1}{(jk+l)^{n}},~\hat{S}(n,k,l)=\sum_{j=-%
\infty }^{\infty }\frac{(-1)^{j}}{(jk+l)^{n}}$ were analyzed and some of
them were calculated. From the results of this paper it follows that the
following sums: 
\begin{equation*}
\mathbb{M}_{k}^{(m,i)}+(-1)^{k+1}\mathbb{M}_{k}^{(m,m-i)}
\end{equation*}%
have values of the form $\pi ^{k}$ times some know, analytic number. Notice
that this statement is trivial for $k$ odd, $m$ even and $i\allowbreak
=\allowbreak m/2.$ 

In particular we get for $k\allowbreak =\allowbreak 2l$ we have $\mathbb{M}%
_{2l}^{(m,i)}\allowbreak -\allowbreak \mathbb{M}_{2l}^{(m,m-i)}\allowbreak
=\allowbreak \frac{1}{m^{2l}}(\zeta (2l,l/(2m))\allowbreak -\allowbreak
\zeta (2l,(m+i)/(2m))\allowbreak \allowbreak -\allowbreak \mathbb{\zeta (}%
2l,(m-i)/(2m))\allowbreak +\allowbreak \zeta (2l,(2m-i)/(2m))\allowbreak
=\allowbreak \frac{1}{m^{2l}}(\zeta (2l,l/(2m))\allowbreak +\allowbreak
\allowbreak \zeta (2l,1-i/(2m))\allowbreak -\allowbreak \zeta
(2l,(m+i)/(2m)-\zeta (2l,(m-i)/(2m)).$

Following \cite{Szab14} we also have for $k\geq 1$:%
\begin{equation*}
S(2k,4,1)=\frac{1}{4^{2k}}(\zeta (2k,1/4)+\zeta (2k,3/4))\allowbreak
=\allowbreak \pi ^{2k}\frac{(2^{2k}-1)}{2(2k)!}(-1)^{k+1}B_{2k},
\end{equation*}%
where $B_{2k}$ denotes $2k-th$ Bernoulli number. In particular we have 
\begin{equation*}
16\mathbb{K=}(\zeta (2,1/4)-\zeta (2,3/4));(\zeta (2,1/4)+\zeta
(2,3/4))\allowbreak =\allowbreak 2\pi ^{2}.
\end{equation*}

Finally let us recall that $\zeta (2l,1)\allowbreak =\allowbreak \allowbreak
(-1)^{l+1}B_{2l}\frac{(2\pi )^{2l}}{2(2l)!}.$ Using formula (\ref{riem}) we
get:%
\begin{equation*}
\allowbreak (-1)^{l+1}B_{2l}\frac{(2\pi )^{2l}}{2(2l)!}\allowbreak
=\allowbreak \frac{2^{2l-1}}{2^{2l-1}-1}\sum_{n=1}^{\infty }\frac{%
H_{n}^{(2l-1)}}{n2^{n}},
\end{equation*}%
and consequently we obtain the following expansions of even powers of $\pi :$%
\begin{equation*}
\pi ^{2l}\allowbreak =\allowbreak (-1)^{l+1}\frac{(2l)!}{(2^{2l-1}-1)B_{2l}}%
\sum_{n=1}^{\infty }\frac{H_{n}^{(2l-1)}}{n2^{n}}.
\end{equation*}

\end{document}